\documentclass{amsart}
\usepackage{amscd,amsmath,amsthm,amssymb,enumerate}
\usepackage[left]{lineno}
\usepackage{pstricks}
\usepackage[utf8]{inputenc}
\usepackage{epstopdf}

\usepackage{comment}

\usepackage[all]{xy}

 \def\NZQ{\mathbb}               
 \def\NN{{\NZQ N}}

 \def\G{{\mathcal G}}

 \def\opn#1#2{\def#1{\operatorname{#2}}} 
 \opn\chara{char} \opn\length{\ell} \opn\pd{pd} \opn\rk{rk}
 \opn\projdim{proj\,dim} \opn\injdim{inj\,dim} \opn\rank{rank}
 \opn\depth{depth} \opn\grade{grade} \opn\height{height}
 \opn\embdim{emb\,dim} \opn\codim{codim}
 
 \opn\Tr{Tr} \opn\bigrank{big\,rank}
 \opn\superheight{superheight}\opn\lcm{lcm}
 \opn\trdeg{tr\,deg}
 \opn\reg{reg} \opn\lreg{lreg} \opn\ini{in} \opn\lpd{lpd}
 \opn\size{size} \opn\sdepth{sdepth}
 \opn\link{link}\opn\fdepth{fdepth}\opn\lex{lex}

 \opn\div{div} \opn\Div{Div} \opn\cl{cl} \opn\Cl{Cl}
 \opn\Spec{Spec} \opn\Supp{Supp} \opn\supp{supp} \opn\Sing{Sing}
 \opn\Ass{Ass} \opn\Min{Min}\opn\Mon{Mon}
 \opn\Ann{Ann} \opn\Rad{Rad} \opn\Soc{Soc}
 \opn\Im{Im} \opn\Ker{Ker} \opn\Coker{Coker} \opn\Am{Am}
 \opn\Hom{Hom} \opn\Tor{Tor} \opn\Ext{Ext} \opn\End{End}
 \opn\Aut{Aut} \opn\id{id}
 
 \opn\nat{nat}
 \opn\pff{pf}
 \opn\Pf{Pf} \opn\GL{GL} \opn\SL{SL} \opn\mod{mod} \opn\ord{ord}
 \opn\Gin{Gin} \opn\Hilb{Hilb}\opn\sort{sort}
 \opn\aff{aff} \opn
 \con{conv} \opn\relint{relint} \opn\st{st}
 \opn\lk{lk} \opn\cn{cn} \opn\core{core} \opn\vol{vol}  \opn\inp{inp} \opn\nilpot{nilpot}
 \opn\link{link} \opn\star{star}\opn\lex{lex}\opn\set{set}
 \opn\width{wd}
 \opn\ecart{ecart}
 \opn\gr{gr}
 
 \def\pot#1#2{#1[\kern-0.28ex[#2]\kern-0.28ex]}

 \opn\dirlim{\underrightarrow{\lim}}
 \opn\inivlim{\underleftarrow{\lim}}
 \let\to=\rightarrow
 
 \def\Implies{\ifmmode\Longrightarrow \else
         \unskip${}\Longrightarrow{}$\ignorespaces\fi}
 \def\implies{\ifmmode\Rightarrow \else
         \unskip${}\Rightarrow{}$\ignorespaces\fi}
 \def\iff{\ifmmode\Longleftrightarrow \else
         \unskip${}\Longleftrightarrow{}$\ignorespaces\fi}

 \let\:=\colon

 \def\Soc{{\mathbf Soc}}

 \def\opn#1#2{\def#1{\operatorname{#2}}} 
 \opn\chara{char} \opn\length{\ell} \opn\pd{pd} \opn\rk{rk}
 \opn\projdim{proj\,dim} \opn\injdim{inj\,dim} \opn\rank{rank}
 \opn\depth{depth} \opn\grade{grade} \opn\height{height}
 \opn\bigheight{bigheight}
 \opn\embdim{emb\,dim} \opn\codim{codim}
 
 \opn\superheight{superheight}\opn\lcm{lcm}
 \opn\trdeg{tr\,deg}
 \opn\reg{reg} \opn\lreg{lreg} \opn\ini{in} \opn\lpd{lpd}
 \opn\size{size} \opn\sdepth{sdepth}
 \opn\link{link}\opn\fdepth{fdepth}\opn\lex{lex}
 \opn\type{type}
 \opn\gap{gap}
 \opn\arithdeg{arith-deg}
 \opn\Deg{Deg}
 \opn\sat{sat}
 \opn\mat{mat}
 \opn\Mat{Mat}
 \opn\div{div} \opn\Div{Div} \opn\cl{cl} \opn\Cl{Cl}
 \opn\Spec{Spec} \opn\Supp{Supp} \opn\supp{supp} \opn\Sing{Sing}
 \opn\Ass{Ass} \opn\Min{Min}\opn\Mon{Mon} \opn\Max{Max}
 \opn\Ann{Ann} \opn\Rad{Rad} \opn\Soc{Soc}
 \opn\Im{Im} \opn\Ker{Ker} \opn\Coker{Coker} \opn\Am{Am}
 \opn\Hom{Hom} \opn\Tor{Tor} \opn\Ext{Ext} \opn\End{End}
 \opn\Aut{Aut} \opn\id{id}
 
 \opn\nat{nat}
 \opn\pff{pf}
 \opn\Pf{Pf} \opn\GL{GL} \opn\SL{SL} \opn\mod{mod} \opn\ord{ord}
 \opn\Gin{Gin} \opn\Hilb{Hilb}\opn\sort{sort}
 \opn\PF{PF}\opn\Ap{Ap}
 \opn\mult{mult}
 \opn\bight{bight}
 \opn\aff{aff}
 \opn\relint{relint} \opn\st{st}
 \opn\lk{lk} \opn\cn{cn} \opn\core{core} \opn\vol{vol}  \opn\inp{inp} \opn\nilpot{nilpot}
 \opn\link{link} \opn\star{star}\opn\lex{lex}\opn\set{set}
 \opn\width{wd}
 \opn\Fr{F}
 \opn\QF{QF}
 \opn\G{G}
 \opn\type{type}\opn\res{res}
 \opn\conv{conv}
 \opn\Shad{Shad}
 \opn\gr{gr}
 
 \def\pot#1#2{#1[\kern-0.28ex[#2]\kern-0.28ex]}

 \opn\dirlim{\underrightarrow{\lim}}
 \opn\inivlim{\underleftarrow{\lim}}
 \let\to=\rightarrow
 
 \def\Implies{\ifmmode\Longrightarrow \else
         \unskip${}\Longrightarrow{}$\ignorespaces\fi}
 \def\implies{\ifmmode\Rightarrow \else
         \unskip${}\Rightarrow{}$\ignorespaces\fi}
 \def\iff{\ifmmode\Longleftrightarrow \else
         \unskip${}\Longleftrightarrow{}$\ignorespaces\fi}

 \let\:=\colon
\theoremstyle{plain}
\newtheorem{theorem}{Theorem}[section]

\newtheorem{thm}[theorem]{Theorem}

\newtheorem{prop}[theorem]{Proposition}

\newtheorem{cor}[theorem]{Corollary}

\newtheorem{lem}[theorem]{Lemma}
\newtheorem{claim}{Claim}

\theoremstyle{definition}

\newtheorem{ex}[theorem]{Example}

\newtheorem{quest}[theorem]{Question}

\newtheorem{rem}[theorem]{Remark}
\newtheorem{notation}[theorem]{Notation}

\newtheorem{fact}[theorem]{Fact}
\newtheorem*{acknowledgments}{Acknowledgments}

\newtheorem*{obs}{Observation}

 \let\epsilon\varepsilon
 \let\kappa=\varkappa
 \def\qed{\ifhmode\textqed\fi
       \ifmmode\ifinner\quad\qedsymbol\else\dispqed\fi\fi}
 \def\textqed{\unskip\nobreak\penalty50
        \hskip2em\hbox{}\nobreak\hfil\qedsymbol
        \parfillskip=0pt \finalhyphendemerits=0}
 \def\dispqed{\rlap{\qquad\qedsymbol}}
 
 \opn\dis{dis}
 \def\pnt{{\raise0.5mm\hbox{\large\bf.}}}
 
 \opn\Lex{Lex}

\newcommand{\rme}{\mathrm{e}}

\newcommand{\rmr}{\mathrm{r}}

\newcommand{\rmH}{\mathrm{H}}

\newcommand{\rmQ}{\mathrm{Q}}

\newcommand{\calH}{\mathcal{H}}

\newcommand{\calP}{\mathcal{P}}

\newcommand{\fkc}{\mathfrak{c}}

\newcommand{\fkm}{\mathfrak{m}}
\newcommand{\fkn}{\mathfrak{n}}

\newcommand{\fkp}{\mathfrak{p}}
\newcommand{\fkq}{\mathfrak{q}}

\newcommand{\fkM}{\mathfrak{M}}

\def\ol{\overline}

\def\tr{\mathrm{tr}}
\def\Tr{\mathrm{Tr}}
\def\regTr{\mathrm{reg}\text{-}\mathrm{Tr}}
\def\Rad{\mathrm{Rad}}

\title{When are trace ideals finite?}

\author{Shinya Kumashiro}
\address{Shinya Kumashiro: National Institute of Technology (KOSEN), Oyama College
771 Nakakuki, Oyama, Tochigi, 323-0806, Japan}
\email{skumashiro@oyama-ct.ac.jp}

\thanks{2020 {\em Mathematics Subject Classification.} 13C05, 13C13, 13H10, 13B22, 13J10}
\thanks{{\em Key words and phrases.} trace ideal, Noetherian ring, Cohen-Macaulay ring, complete ring, canonical module, minimal multiplicity}
\thanks{The author was supported by JSPS KAKENHI Grant Number 21K13766 and by Grant for Basic Science Research Projects from the Sumitomo Foundation (Grant number 2200259).}

\begin{document}

\begin{abstract}
In this paper, we study Noetherian local rings $R$ having a finite number of trace ideals. We proved that such rings are of dimension at most two. Furthermore, if the integral closure of $R/H$, where $H$ is the zeroth local cohomology, is equi-dimensional, then the dimension of $R$ is at most one. In the one-dimensional case, we can reduce to the situation that rings are Cohen-Macaulay. Then, we give a necessary condition to have a finite number of  trace ideals in terms of the value set obtained by the canonical module. We also gave the correspondence between trace ideals of $R$ and those of the endomorphism algebra of the maximal ideal of $R$ when $R$ has minimal multiplicity. 
\end{abstract}

\maketitle



\section{Introduction}\label{section1}

Let $R$ be a commutative Noetherian ring, and let $M$ be a finitely generated $R$-module. Then
\begin{align*} 
\mathrm{tr}_R(M)=\sum_{f\in \Hom_R(M, R)} \Im f 
\end{align*}
is called the {\it trace ideal} of $M$. An ideal $I$ in $R$ is called a {\it trace ideal} if $I=\mathrm{tr}_R(M)$ for some $R$-module $M$. 
The notion of trace ideals is useful to study the structure of reflexive modules. Suppose that $M$ is a faithful reflexive $R$-module. Then, the endomorphism algebra $T(M)=\Hom_R(\tr_R(M), \tr_R(M))$ is the center of $\Hom_R(M, M)$ (\cite[Introduction]{Lin}). It is also known that $M$ can be regarded as a $T(M)$-module  (\cite[(7.2) Proposition]{Ba}, \cite[Proposition 2.4]{IK}). 
Based on these results, 
Isobe and Kumashiro, and independently Dao, provided a certain direct-sum decomposition for reflexive modules over (one-dimensional) Arf local rings (\cite[Theorem A]{D} and \cite[Theorem 1.1]{IK}). 

The notion of trace ideals is not only useful for the study of a given module, but is also deeply related to the property of rings. Indeed, Kobayashi and Takahashi characterized Noetherian local rings such that all ideals are isomorphic to some trace ideals (\cite[Corollary 1.4]{KT}). Herzog and Rahimbeigi proved that for one-dimensional analytically irreducible Gorenstein local $K$-algebras, where $K$ is an infinite field, the finiteness of trace ideals is equivalent to the finiteness of indecomposable maximal Cohen-Macaulay modules up to isomorphism (\cite[Corollary 2.16]{HM}). 

In addition, for a Cohen-Macaulay local ring having the canonical module, it is also known that the trace ideal of the canonical module defines the non-Gorenstein locus of the ring. With this reason the trace ideal of the canonical module is utilized in the study of non-Gorenstein Cohen-Macaulay rings, as {\it nearly Gorenstein rings} (\cite{DKT, HHS}).

Among these subjects, in this paper, we study the following finiteness problem on trace ideals, which is already posed by several papers \cite[Question 7.16(1)]{DMS}, \cite[Question 3.7]{F}, and \cite{HM}. 

\begin{quest} \label{question}
When do Noetherian local rings have a finite number of trace ideals?
\end{quest}

Question \ref{question} may be arised from the study of classification of isomorphism classes of maximal Cohen-Macaulay modules (see \cite[Question 3.7]{F}).
Noting that distinct trace ideals are non-isomorphic (this is immediately observed from Fact \ref{remtrace}(a)), to see how many non-isomorphic trace ideals there are, we only need to know what is the set of trace ideals. We should mention a previous study \cite{GIK2} on the set of trace ideals.

With Question \ref{question}, this paper proceeds as follows. Let $R$ be a Noetherian local ring. We denote by $H$ the zeroth local cohomology of $R$. $\ol{*}$ denotes the integral closure of $*$. 

We first obtain that the finiteness of trace ideals forces the dimension of rings to be at most two. If $\ol{R/H}$ is equi-dimensional, that is, all maximal ideals of $\ol{R/H}$ have the same height, then the dimension must be at most one (Theorem \ref{thm1}). For the one-dimensional case, we can reduce to the situation that rings are Cohen-Macaulay by passing to the residue ring $R/H$ (Proposition \ref{lem1}). 
Then, we obtain a necessary condition of the finiteness of trace ideals in terms of the value set obtained by the canonical module (Theorem \ref{mainthm2}). 
In addition, if $R$ has minimal multiplicity, we give a one-to-one correspondence between the set of all trace ideals of $R$ (excluding $R$) and those of $\Hom_R(\fkm, \fkm)$ (Theorem \ref{main3}). As a corollary, analytically irreducible Arf local rings have finite trace ideals (Corollary \ref{cor4.11}). 
We also characterize when the set of all trace ideals is as small as possible in dimension at most one (Theorem \ref{aa3.3}).

\begin{notation}
Throughout this paper, let $R$ be a Noetherian ring, and let $I$ be an ideal in $R$. The following are well-known notions used in this paper.
\begin{itemize}
\item $\Rad (I)$ denotes the {\it radical} of $I$ in the sense of \cite[page 8]{AM}.
\item $J(R)$ denotes the {\it Jacobson radical} of $R$.
\item $\grade (I, R)$ denotes the {\it grade} of $I$ on $R$ in the sense of \cite[Definition 1.2.6]{BH}. An ideal $I$ with $\grade(I, R)>0$, i.e., an ideal $I$ containing a non-zerodivisor of $R$, is called a {\it regular ideal} in $R$.
\item Let $\rmQ(R)$ denote the {\it total ring of fractions} of $R$, and let $\ol{R}$ denote the {\it integral closure} of $R$. Then a finitely generated $R$-submodule of $\rmQ(R)$ containing a non-zerodivisor of $R$ is called a {\it fractional ideal}. For two fractional ideals $I$ and $J$, $I:J$ denotes the (fractional) colon ideal of $I$ and $J$ which is given by the set $\{\alpha\in \rmQ(R) \ \mid \  \alpha J\subseteq I\}$. It is known that $I:J\cong \Hom_R(J, I)$ by the corresponding of $\alpha \mapsto \hat{\alpha}$, where $\hat{\alpha}$ denotes the multiplication map by $\alpha$  (see \cite[page 17]{HK}). We denote by $I:_R J$ the restriction of $I:J$ in $R$, that is, $I:_RJ=(I:J)\cap R$.
\end{itemize}
\end{notation}

\begin{acknowledgments}
Question \ref{question}, the starting point of this paper, was asked by J\"urgen Herzog to the author. The author is grateful to him. The author thanks Ryotaro Isobe and Kazuho Ozeki for giving useful comments to Proposition \ref{lem1}, Lemma \ref{lem3}, and Theorem \ref{thm1}. The author also thanks Toshinori Kobayashi for telling the author about Kunz's coordinates (Remark \ref{remrem5.3}). The author is also grateful to the anonymous referee for his/her careful reading and for pointing out the error in the previous version.
\end{acknowledgments}

\section{When are ideals trace ideals?}\label{section2}

Let $R$ be a Noetherian ring, and let $I$ be an ideal in $R$. Let 
\begin{align*}
\Tr(R) &= \{\text{trace ideals in $R$}\} \quad \text{and}\\
\regTr(R) &= \{\text{regular trace ideals in $R$}\}.
\end{align*}
First, we note fundamental facts on trace ideals.

\begin{fact}\label{remtrace} {\rm (\cite[Example 2.4 and Proposition 2.8]{Lin})}
\begin{enumerate}[{\rm (a)}] 
\item  For any ideal $I$, $I\subseteq \tr_R(I)$. $I$ is a trace ideal if and only if $I=\tr_R(I)$. This is equivalent to saying that $\Im f\subseteq I$ for all $f\in \Hom_R(I, R)$, i.e., $\Hom_R(I, R)=\Hom_R(I, I)$. 
\item If $\grade (I, R)\ge 2$, then $I$ is a trace ideal in $R$.
\item Let $M$, $N$ be finitely generated $R$-modules. Then $\tr_R(M\oplus N)=\tr_R(M) + \tr_R(N)$.
\item Let $(R, \fkm)$ be a Noetherian local ring, and let $M$ be an $R$-module. Then $\tr_R(M)=R$ if and only if $M$ has a free summand.
\end{enumerate}
\end{fact}

The following characterizes when prime ideals are trace ideals. Note that this result may be known, although we could not find any references. 

\begin{prop} \label{proptrace}
Let $\fkp$ be a prime ideal in $R$. Then, $\fkp$ is a trace ideal in $R$ if and only if $R_\fkp$ is not a discrete valuation ring. 
\end{prop}

\begin{proof} 
$\fkp$ is a trace ideal if and only if  $\fkp=\tr_R (\fkp)$ by Fact \ref{remtrace}(a). Since $\tr_R (\fkp)/\fkp\subseteq R/\fkp$, we obtain that $\Ass_R \tr_R (\fkp)/\fkp \subseteq \{\fkp\}$. Hence, 
\begin{align*} 
&\text{$\fkp$ is a trace ideal} && \Leftrightarrow \quad  (\tr_R (\fkp)/\fkp)_\fkp=0 \quad \Leftrightarrow \quad  \tr_{R_\fkp} (\fkp R_\fkp)/\fkp R_\fkp =0 \\
 \Leftrightarrow \quad & \text{$\fkp R_\fkp$ is a trace ideal in $R_\fkp$} && \Leftrightarrow \quad  \text{$R_\fkp$ is not a discrete valuation ring,} 
\end{align*} 
where the second equality follows from \cite[Proposition 2.8(viii)]{Lin} and the fourth equality follows from the following claim.
\end{proof}

\begin{claim} \label{claim1} {\rm (\cite[Proposition 1.8]{HM})}
Let $(R, \fkm)$ be a Noetherian local ring. Then $\fkm$ is not a trace ideal if and only if $R$ is a discrete valuation ring. 
\end{claim}

\begin{proof}[Proof of Claim \ref{claim1}]
Since $\fkm \subseteq \tr_R(\fkm)$, $\fkm$ is not a trace ideal if and only if $\tr_R(\fkm)=R$. By Fact \ref{remtrace}(d), this is equivalent to saying that $\fkm$ has a free summand. In other words, $\fkm=(a) + I$ and $(a)\cap I=0$ for some non-zerodivisor $a$ of $R$ and some ideal $I$ in $R$. It follows that $aI\subseteq (a)\cap I=0$; hence, $I=0$. Therefore, $\fkm$ is not a trace ideal if and only if $\fkm =(a)$ for some non-zerodivisor $a$ of $R$.
\end{proof}

Let $(R, \fkm)$ be a Noetherian local ring of dimension $>0$. Let 
\[
H=\bigcup_{n>0} (0):_R \fkm^n\cong \rmH_\fkm^0(R)
\] 
denote the zeroth local cohomology of $R$. By applying the local cohomology functor $\rmH_\fkm^* (-)$ to 
$0 \to H \to R \to R/H \to 0$, we obtain that $0\to H \xrightarrow{\cong} H \to \rmH_\fkm^0(R/H) \to 0$.
Hence, $\depth R/H>0$. 
The following proposition is effective for attributing to rings with positive depth.

\begin{prop} \label{lem1}
Let $(R, \fkm)$ be a Noetherian local ring of dimension $>0$. Let $H$ denote the zeroth local cohomology of $R$. Suppose that $I$ is an ideal in $R$ such that $I\supseteq H$. If $I/H$ is a trace ideal in $R/H$ containing a non-zerodivisor of $R/H$, then $I$ is a trace ideal in $R$. In particular, if $\Tr(R)$ is finite, then 
 $\regTr(R/H)$ is also finite.
\end{prop}

\begin{proof}
If $H=0$, then there is nothing to prove. Suppose that $H\ne 0$. 
Since $I$ contains a non-zerodivisor of $R/H$, it follows that $\Hom_{R} (H/IH, R/H)=0$.
By applying the functor $\Hom_{R}(-, R/H)$ to $0 \to H/IH \to I/IH \xrightarrow{\pi} I/H \to 0$, we obtain that 
\begin{align}\label{great}
\Hom_{R}(I/H, R/H) \xrightarrow{\pi^*} \Hom_{R} (I/IH, R/H)
\end{align}
is an isomorphism.

Now, let $f\in \Hom_R(I, R)$. Then, the composition $I\xrightarrow{f} R \to R/H$ induces the map $g: I/IH \to R/H$. By the isomorphism \eqref{great}, $g$ factors through $\pi$, that is, there exists $h\in \Hom_{R} (I/H, R/H)$ such that $h\circ \pi=g$. Therefore, we obtain that
\[
[\Im f + H]/H=\Im g = \Im (h\circ \pi)=\Im h\subseteq \tr_{R/H}(I/H)=I/H.
\]
Hence, $\Im f\subseteq I$ for all $f\in \Hom_R(I, R)$. This proves that $I$ is a trace ideal in $R$.
\end{proof}

\begin{lem} \label{lem24}
\label{lem2}
Let $R$ be a Noetherian ring. Then 
\[
\ol{R}=\bigcup_{I\in \regTr(R)} I:I.
\]
In particular, if $\regTr(R)$ is finite, then $\ol{R}$ is finitely generated as an $R$-module.
\end{lem}

\begin{proof}
($\supseteq$): Note that $I:I$ is a subring of $\rmQ(R)$ and finitely generated as an $R$-module since $I:I\cong \Hom_R(I, I)$. Hence, $I:I$ is a subring of $\ol{R}$.

($\subseteq$): By noting that $R$ is a trace ideal in $R$ itself, the right hand side of the equation contains $R=R:R$.
Let $\alpha\in \ol{R}\setminus R$, and set $I=R:R[\alpha]$. Then $I\subsetneq R$ is a trace ideal in $R$ because $I=(R:R[\alpha])R[\alpha]=\tr_R(R[\alpha])$. Hence, because $IR[\alpha]=I$, we obtain that $\alpha\in R[\alpha]\subseteq I:I$.
\end{proof}

\begin{lem} \label{lem3}
Let $R$ be a Noetherian ring, and let $I$ be a regular ideal in $R$. Let $B$ be an intermediate ring between $R$ and $\ol{R}$. 
If $IB$ is a trace ideal in $B$ and $IB\cap R=I$, then $I$ is a regular trace ideal in $R$. 
\end{lem}

\begin{proof}
Since $R:I\subseteq B: IB$, we obtain that 
\[
I\subseteq \tr_R(I)=(R:I)I\subseteq (B: IB)IB =IB.
\]
Hence, $I\subseteq \tr_R(I) \subseteq IB\cap R=I$.
\end{proof}

Now, we can prove the first main theorem of this paper.

\begin{thm} \label{thm1}
Let $(R, \fkm)$ be a Noetherian local ring. Suppose that $\Tr(R)$ is finite. Then the following holds true.
\begin{enumerate}[\rm(a)] 
\item $\dim R\le 2$ and $\ol{R}$ is finitely generated as an $R$-module.
\item If $\ol{R/H}$ is equi-dimensional (that is, all maximal ideals of $\ol{R/H}$ have the same height), where $H$ is the zeroth local cohomology, then $\dim R \le 1$.
\end{enumerate}
\end{thm}

\begin{proof}
(a): Suppose that there exists a prime ideal $\fkq$ of height $3$. Then there are infinitely many prime ideals contained in $\fkq$ such that the heights of prime ideals are $2$. Indeed, assume that there are only finite prime ideals $\fkp_1, \fkp_2, \dots, \fkp_n$ contained in $\fkq$ such that $\height_R \fkp_i=2$. Then, by the prime avoidance theorem, there exists an element $a\in \fkq \setminus \bigcup_{i=1}^{n} \fkp_i$. This proves that $\height_{R_\fkq} (a)= 3$. This contradicts the Krull's height theorem. Therefore, if $\Tr(R)$ is finite, then there is no prime ideal of height $3$ by Proposition \ref{proptrace}. Thus $\dim R\le 2$. 

The fact that $\ol{R}$ is finitely generated as an $R$-module is obtained by Lemma \ref{lem24}.

(b): Suppose that there exists a Noetherian local ring $(R, \fkm)$ of dimension $2$ such that $\Tr(R)$ is finite. 
Let $A=R/H$ and $\fkm_A=\fkm/H$, where $H$ denotes the zeroth local cohomology of $R$. Note that $\dim A=2$ and $\depth A>0$. 
By Proposition \ref{lem1}, $\regTr(A)$ is finite. 
$\ol{A}$ is finitely generated as an $A$-module by Lemma \ref{lem2}. 
Let $J(\ol{A})$ denote the Jacobson radical of $\ol{A}$.
By applying the $\ol{A}$-dual to 
\[
0 \to J(\ol{A}) \xrightarrow{\iota} \ol{A} \to \ol{A}/J(\ol{A}) \to 0,
\] 
we obtain that 
\begin{align}
\begin{split} \label{eq1234}
0\to& \Hom_{\ol{A}} (\ol{A}/J(\ol{A}), \ol{A}) \to \Hom_{\ol{A}} (\ol{A}, \ol{A}) \xrightarrow{\iota^*}  \Hom_{\ol{A}} (J(\ol{A}), \ol{A}) \\
\to& \Ext_{\ol{A}}^1(\ol{A}/J(\ol{A}), \ol{A}) \to 0. 
\end{split}
\end{align}
Since $\ol{R}$ is finitely generated as an $R$-module, $J(\ol{A})$ contains a non-zerodivisor of $\ol{A}$. It follows that the map $\iota^*$ is identified by the inclusion $\ol{A} \to \ol{A}:J(\ol{A})$. We prove that $J(\ol{A})$ is a trace ideal in $\ol{A}$. Indeed, $\tr_{\ol{A}}(J(\ol{A}))/J(\ol{A})$ is of finite length and $(\tr_{\ol{A}}(J(\ol{A}))/J(\ol{A}))_\fkn=\tr_{\ol{A}_\fkn}(\fkn \ol{A}_\fkn)/\fkn \ol{A}_\fkn$ for all $\fkn\in \Max \ol{A}$. By the assumption, each $\ol{A}_\fkn$ is of dimension $2$ and thus it is not a discrete valuation ring; hence, $\tr_{\ol{A}_\fkn}(\fkn \ol{A}_\fkn)/\fkn \ol{A}_\fkn=0$ by Proposition \ref{proptrace}. Therefore, $\tr_{\ol{A}}(J(\ol{A}))/J(\ol{A})=0$. 
Since $J(\ol{A})$ is a trace ideal, it follows that $\ol{A}\subseteq \ol{A}:J(\ol{A})=J(\ol{A}):J(\ol{A})\subseteq \ol{A}$. Thus, $\iota^*$ in \eqref{eq1234} is an isomorphism; hence, $\grade(J(\ol{A}), \ol{A})=2$ by \eqref{eq1234}.

Let $J$ be an ideal in $\ol{A}$ such that $J(\ol{A}) \subseteq \Rad(J)$.  Then $J\cap A$ is an $\fkm_A$-primary ideal in $A$. By noting that $\ol{A}$ is finitely generated as an $A$-module, $\ell_{\ol{A}} (\ol{A}/(J\cap A)\ol{A})\le \ell_{A} (\ol{A}/(J\cap A)\ol{A})<\infty$, thus $J(\ol{A}) \subseteq \Rad((J\cap A)\ol{A})$. It follows that $(J\cap A)\ol{A}$ is a trace ideal in $\ol{A}$ by Fact \ref{remtrace}(b).
On the other hand, $(J\cap A)\ol{A} \cap A=J\cap A$ (see \cite[Proposition 1.17(ii)]{AM}). 
Therefore, by applying Lemma \ref{lem3} as $B=\ol{A}$, $\regTr(A)$ includes the set 
\[
S=\{J\cap A\ \mid \ \text{$J$ is an ideal in $\ol{A}$ such that $J(\ol{A}) \subseteq \Rad(J)$}\}.
\]
Let $S=\{J_i\cap A\  : \ 1\le i \le \ell\}$. Then 
\[
\textstyle \left(\bigcap_{i=1}^\ell J_i\right) \cap A\in S.
\]
By noting that $\left(\bigcap_{i=1}^\ell J_i\right) \cap A=\bigcap_{i=1}^\ell \left(J_i \cap A\right)$, there exists the smallest element $J_j\cap A$ in $S$, where $1\le j \le \ell$. In particular, $J_j\cap A\subseteq J(\ol{A})^s \cap A$ for all $s>0$. It follows that 
\[
(J_j \cap A)\ol{A}\subseteq \bigcap_{s>0} J(\ol{A})^s=0
\]
by Krull's theorem (\cite[Corollary 10.19]{AM}). Hence, $J_j \cap A=0$. This contradicts the facts that $\dim A=2$ and $J_j\cap A$ is an $\fkm_A$-primary ideal. Therefore, there is no Noetherian local ring $R$ of dimension $2$ such that $\Tr(R)$ is finite; hence, $\dim R\le 1$.
\end{proof}

We note a remark for the assumption that $\ol{R/H}$ is equi-dimensional in Theorem \ref{thm1}. 

\begin{rem}
Let $K$ be a field, and let $K[[X, Y]]$, $K[[Z]]$ be formal power series rings over $K$.
\begin{enumerate}[\rm(a)] 
\item Set $R=K[[X, X^2Y, X^2Y^2, X^3Y^3]]$. Then $R$ is a Noetherian local domain of dimension two and thus the zeroth local cohomology $H$ is zero. Furthermore, we have $\ol{R}=K[[X, XY]]$ and hence $\ol{R}$ is a local ring. Therefore, by Theorem \ref{thm1}(b), $\Tr(R)$ is an infinite set.
\item Set 
\begin{center}
$S=K[[X, Y]]\times K[[Z]]$ \quad and \quad $R=K(1,1) + J(S) \subseteq S$, 
\end{center}
where $J(S)=\{(a,b) \mid a\in (X, Y), b\in (Z)\}$ is the Jacabson radical of $S$. Then, $R$ is a Noetherian local ring with positive depth since $(X, Z)$ is a non-zerodivisor of $R$. Furthermore, $\ol{R}=S$ since $(1,0)^2-(1,1)(1,0)=0$. It follows that $\ol{R}$ is not equi-dimensional. However, one can also check that $((X, 0), (Y, 0))^s + ((0,Z))\in \Tr(R)$ for all $s>0$, thus $\Tr(R)$ is  infinite.
\end{enumerate}
\end{rem}

The following is an example of a (one-dimensional) non-Cohen-Macaulay local ring having finite trace ideals.

\begin{ex} \label{ex1}
Let $R=K[[X, Y]]/(XY, Y^2)$, where $K$ is a field and $K[[X, Y]]$ is a formal power series ring over $K$. Then $R$ is a Noetherian local ring of dimension $1$ and of depth $0$. Furthermore, we obtain that 
\[
\Tr(R)=\{0, (y), (x, y), R\},
\]
where $x$ and $y$ denote the image of $X$ and $Y$ into $R$, respectively.
\end{ex}

\begin{proof} 
It is easy to see that $R$ is of dimension $1$ and of depth $0$. 
Note that $(y)=(0):_R~\fkm=(0):_R x$. Hence, the image of every $R$-linear homomorphism $f\in \Hom_R((y), R)$ is in $(y)$. Hence, $(y)\in \Tr(R)$. $(x, y)\in \Tr(R)$ follows from Proposition \ref{proptrace}. It is clear that $0$ and $R$ are in $\Tr(R)$.

Suppose that there exists a trace ideal $I$ such that $I\not\in \{0, (y), (x, y), R\}$. Then $I\supseteq (y)$ because 
\[
I \twoheadrightarrow I/\fkm I \twoheadrightarrow R/\fkm \cong (y)\subseteq R.
\] 
Thus, $I/(y)$ is a nonzero ideal in a discrete valuation ring $R/(y)$; hence, $I/(y) \cong R/(y)$. Therefore, we obtain an $R$-linear homomorphism
\[
\varphi: I \twoheadrightarrow I/(y) \cong R/(y) \xrightarrow{\hat{x}} R,
\]
where $\hat{x}$ is a multiplication map by $x$, and thus the image of $\varphi$ is $(x)$. It follows that $I=\tr_R(I)\supseteq (x)$. Hence, $I$ is either $(x, y)$ or $R$. This is a contradiction.
\end{proof}

\section{Smallest trace ideals in dimension $\le 1$}\label{section3}

Due to Theorem \ref{thm1}, in what follows, we focus on the case of dimension $\le 1$. In this case, there exists the smallest trace ideal (excluding the zero ideal). As a result, we obtain a characterization of when the number of elements in $\Tr(R)$ is at most $3$ (Theorem \ref{aa3.3}).

\begin{lem}\label{aa2.1}
Let $(R, \fkm)$ be an Noetherian local ring of depth $0$. Then the socle $\Soc R=(0):_R \fkm$ of $R$ is a trace ideal, and every nonzero trace ideal in $R$ contains $\Soc R=(0):_R \fkm$.
\end{lem}

\begin{proof}
It is clear that for $f\in \Hom_R(\Soc R, R)$, $\Im f\subseteq \Soc R$. Hence, $\Soc R$ is a trace ideal.

Let $I$ be a nonzero trace ideal in $R$. Note that we have a surjection $I\twoheadrightarrow I/\fkm I \twoheadrightarrow R/\fkm$. Hence, for any $0\ne x\in \Soc R$, we have a map $I \to R/\fkm \cong Rx \hookrightarrow R$. Hence $x\in \tr_R(I)=I$.
\end{proof}

\begin{lem}\label{aa2.2}
Let $(R, \fkm)$ be a Cohen-Macaulay local ring of dimension $1$ such that the residue field $R/\fkm$ is infinite. Suppose that $\ol{R}$ is finitely generated as an $R$-module. Then, the conductor $R:\ol{R}$ of $\ol{R}$ is a regular trace ideal, and every regular trace ideal in $R$ contains $R:\ol{R}$.
\end{lem}

\begin{proof}
Since $R:\ol{R}=(R:\ol{R})\ol{R}=\tr_R(\ol{R})$, $R:\ol{R}$ is a trace ideal. $R:\ol{R}$ contains a non-zerodivisor of $R$ since $\ol{R}$ is finitely generated as an $R$-module.

Let $I\in \regTr(R)$. Because $R/\fkm$ is infinite, there exists a non-zerodivisor $a\in I$ such that $I^{n+1}=aI^n$ for some $n>0$. This is equivalent to saying that $(a)\subseteq I \subseteq \ol{(a)}=a\ol{R}\cap R$, where $\ol{(a)}$ denotes the integral closure of $(a)$ (\cite[Corollary 1.2.5]{SH}). It follows that 
\[
R\subseteq \tfrac{I}{a}=\{x/a\in \rmQ(R) \ \mid \ x\in I\} \subseteq \ol{R}.
\]
Hence, we obtain that 
\begin{align*}
I=\tr_R(I)=(R:I)I=\left(R:\tfrac{I}{a}\right)\tfrac{I}{a}\supseteq R:\tfrac{I}{a}\supseteq R:\ol{R}.
\end{align*}
\end{proof}

\begin{thm}\label{aa3.3}
Let $(R, \fkm)$ be a Noetherian local ring of dimension $\le 1$ such that $R/\fkm$ is infinite. Then the following are equivalent:
\begin{enumerate}[{\rm (a)}] 
\item $\Tr(R)\subseteq \{0, \fkm, R\}$.
\item $R$ satisfies either one of the following:
\begin{enumerate}[{\rm (i)}] 
\item $R$ is an Artinian ring having minimal multiplicity, i.e., $\fkm^2=0$. 
\item $R$ is a Cohen-Macaulay ring of dimension one, analytically unramified, and satisfies $\fkm\subseteq R:\ol{R}$. 
\end{enumerate}
\end{enumerate}
In particular, if $\Tr(R)\subseteq \{0, \fkm, R\}$, then $R$ is a Cohen-Macaulay local ring having minimal multiplicity.
\end{thm}

\begin{proof}
(a) $\Rightarrow$ (b): Suppose that $\dim R=0$. Then, by Lemma \ref{aa2.1}, either $\Soc R=\fkm$ or $\Soc R=R$ holds. The former implies that $\fkm^2=0$ and the latter implies that $R$ is a field. 

Suppose that $\dim R=1$. Then $\Soc R$ is neither $\fkm$ nor $R$, thus $\Soc R=0$. Hence, since $\Hom_R(R/\fkm, R)=0$, $R$ is a Cohen-Macaulay ring. Furthermore, Lemma \ref{lem2} shows that $\ol{R}$ is finitely generated as an $R$-module, i.e., $R$ is analytically unramified. Therefore, by Lemma \ref{aa2.2}, either $R=R:\ol{R}$ or $\fkm=R:\ol{R}$ holds. The former says that $R=\ol{R}$, thus $R$ is a discrete valuation ring. In particular, $R$ has minimal multiplicity. The latter says that $\ol{R}\subseteq R:\fkm=\fkm:\fkm \subseteq \ol{R}$. Hence, $R$ also has minimal multiplicity by \cite[Theorem 5.1]{GMP}.

(b) $\Rightarrow$ (a): This follows by Lemmas \ref{aa2.1} and \ref{aa2.2}.
\end{proof}

Theorem \ref{aa3.3} generalizes \cite[Proposition 6.3 (1)]{DMS} and \cite[Proposition 3.5]{GIK2}. Note that the condition (b)(ii) in Theorem \ref{aa3.3} is equivalent to saying that $R$ is nearly Gorenstein and far-flung Gorenstein in the senses of  \cite[Definition 2.2]{HHS} and \cite[Definition 2.3]{HKS2} (see \cite[Remark 2.4]{HKS2}).

\section{Trace ideals in dimension one}\label{section4}

Let $(R, \fkm)$ be a Noetherian local ring of dimension $1$ having a finite number of trace ideals. Then, Example \ref{ex1} shows that $R$ is not necessarily Cohen-Macaulay. On the other hand, $R/H$ is a Cohen-Macaulay local ring whose regular trace ideals are finite, where $H$ denotes the zeroth local cohomology of $R$ (Proposition \ref{lem1}). With this perspective, we next investigate one-dimensional Cohen-Macaulay complete local rings having finite regular trace ideals. Note that this is the situation posed by Dao, Maitra, and Sridhar (\cite[Question 7.16(1)]{DMS}). Furthermore, since all nonzero ideals in one-dimensional Cohen-Macaulay local rings are maximal Cohen-Macaulay modules, this is also the one-dimensional case of Faber's question (\cite[Question 3.7]{F}).

First, we note that it is a necessary assumption that $R$ contains an infinite field to constrain the situation. 
Let $H$ be a numerical semigroup, that is, a submonoid of $\NN_0=\{0, 1, 2, \dots\}$ such that $\NN_0\setminus H$ is finite. Then 
\[
R=K[[H]]=K[[t^h : h\in H]]
\] 
is called the {\it numerical semigroup ring} of $H$, where $K[[t]]$ is the formal power series ring over a field $K$. For numerical semigroup rings, we have the same conclusion as Lemma \ref{aa2.2} without the assumption that the residue field $K$ is infinite:

\begin{fact}{\rm (cf. \cite[Proposition A.1]{HHS2}, \cite[Proposition 2.2]{HM})} \label{fact4.1}
Let $R$ be a numerical semigroup ring. Let $I$ be a nonzero trace ideal in $R$. Then $R:\ol{R} \subseteq I$.
\end{fact}

\begin{cor}
Let $R$ be a numerical semigroup ring. If $K$ is a finite field, then $\Tr(R)(=\regTr(R)\cup \{0\})$ is finite. 
\end{cor}

\begin{proof}
Since $R/(R:\ol{R})$ is a finite set, there are only a finite number of ideals containing $R:\ol{R}$ in $R$. In particular, trace ideals are finite by Fact \ref{fact4.1}.
\end{proof}

In what follows, let $(R, \fkm)$ be a one-dimensional Cohen-Macaulay complete local ring and suppose that $R$ contains an infinite field $K$ as a subring (i.e., $R$ is a $K$-algebra). Suppose that $\ol{R}$ is finitely generated as an $R$-module. 
Then, $R$ has the canonical module $\omega_R$ and we can choose an $\fkm$-primary ideal $\omega$ in $R$ as the canonical module (see \cite[4.6 Theorem]{LW} and \cite[Proposition 3.3.18]{BH}). 
Because $R/\fkm$ is also an infinite field, there exists a non-zerodivisor $a\in \omega$ of $R$ such that $\omega^{n+1}=a\omega^n$ for some $n>0$. 
Furthermore, we have an isomorphism 
\[
\varphi: \ol{R} \xrightarrow{\ \ \cong\ \ } R_1\times R_2 \times \cdots \times R_\ell
\]
of rings, where $R_i$ is a discrete valuation ring for $1\le i\le \ell$. Indeed, since complete local rings are Henselian, $\ol{R}$ is a product of local rings $R_i$ (\cite[1.9 Corollary and A.30 Theorem]{LW}). Because $R_i$ are obtained by the localization at each maximal ideal of $\ol{R}$, $R_i$ are also integrally closed. It follows that all $R_i$ are discrete valuation rings.

Let 
\[
\varphi_i: \ol{R} \xrightarrow{\ \ \varphi \ \ } R_1\times R_2 \times \cdots \times R_\ell \xrightarrow{\ \ \pi_i\ \ } R_i
\] 
for $1 \le i \le \ell$, where $\pi_i$ is a canonical surjection. For an integer $1\le i \le \ell$ and a nonzero element $a\in R_i$, we say that 
\[
v_i(a) = \ell_{R_i} (R_i/aR_i)
\]
is the {\it value} of $a$ in $R_i$. For a finitely generated $R$-submodule $A$ of $\ol{R}$, we call a set
 \[
v_i(A) = \{v_i(\varphi_i(x)) \ \mid \ x\in A\}
\]
of $\NN_0$ the {\it value set} of $A$ with respect to $R_i$. Note that $v_i(A)$ becomes a semigroup if $A$ is a {\it birational extension}, that is, $A$ is an intermediate ring between $R$ and $\ol{R}$ such that $A$ is finitely generated as an $R$-module.

\begin{prop}\label{keysec4} 
In addition to the above assumptions and notation, let $\fkn_i=(f_i)$ denote the maximal ideal of $R_i$ for $1\le i \le \ell$. Choose $t_i\in \ol{R}$ such that $\varphi(t_i)=(0, \dots, 0, f_i, 0, \dots, 0)$. For an integer $n\ge 2$, assume that $1, n, n+1\not\in v_i(\tfrac{\omega}{a})$. 
Then, for elements $k$ and $k'$ in $K$, the following are equivalent:\\
{\rm (a)} \ $R:R[t_i^n+kt_i^{n+1}]=R:R[t_i^n+k't_i^{n+1}]$. \quad \quad
{\rm (b)} \ $k=k'$.
\end{prop}

\begin{proof}
The proof proceeds in a similar way as \cite[Theorem 2.15]{HM}, but we include a proof for the sake of completeness.
It is enough to prove (a) $\Rightarrow$ (b). Suppose that $R:R[t_i^n+kt_i^{n+1}]=R:R[t_i^n+k't_i^{n+1}]$. 
Note that 
\[
R:R[t_i^n+kt_i^{n+1}]=(\tfrac{\omega}{a}:\tfrac{\omega}{a}):R[t_i^n+kt_i^{n+1}]=\tfrac{\omega}{a}:\tfrac{\omega}{a}[t_i^n+kt_i^{n+1}].
\]
Hence, by applying the canonical dual $\tfrac{\omega}{a}: -\cong \Hom_R(-, \omega_R)$ to $R:R[t_i^n+kt_i^{n+1}]=R:R[t_i^n+k't_i^{n+1}]$, we obtain that $\tfrac{\omega}{a}[t_i^n+kt_i^{n+1}]=\tfrac{\omega}{a}[t_i^n+k't_i^{n+1}]$. Write 
\[
t_i^n+kt_i^{n+1}=c_0 + c_1(t_i^n+k't_i^{n+1}) + g,
\]
where $c_0, c_1\in \tfrac{\omega}{a}$ and $g\in \ol{R}$ with $v_i(\varphi_i(g)) \ge n+2$. 
By mapping $\varphi_i(-)$, we obtain that 
\begin{align}\label{eqaaa}
f_i^n+\varphi_i(k) f_i^{n+1}=\varphi_i(c_0) + \varphi_i(c_1)[f_i^n+\varphi_i(k')f_i^{n+1}] + \varphi_i(g).
\end{align}
Thus,
\[
\varphi_i(c_0) =f_i^n+\varphi_i(k) f_i^{n+1} - \varphi_i(c_1)[f_i^n+\varphi_i(k')f_i^{n+1}] - \varphi_i(g);
\]
hence, we have either $\varphi_i(c_0)=0$ or $v_i(\varphi_i(c_0))\ge n$. 
Assume that $v_i(\varphi_i(c_0))\ge n$. Since $c_0\in \tfrac{\omega}{a}$, we have $v_i(\varphi_i(c_0))\ne n, n+1$ by the assumption. Hence, $v_i(\varphi_i(c_0)) \ge n+2$. It follows that we can replace $g$ with $g+c_0$ preserving the condition $v_i(\varphi_i(g)) \ge n+2$. Thus, we may assume that $\varphi_i(c_0)=0$. Then, \eqref{eqaaa} can be rewritten as
\[
(1-\varphi_i(c_1))f_i^n= [\varphi_i(c_1)\varphi_i(k')- \varphi_i(k)]f_i^{n+1} + \varphi_i(g).
\]
Either $1-\varphi_i(c_1)=0$ or $v_i(1-\varphi_i(c_1))\ge 1$ holds. By noting that $\varphi_i(1)=1$, we obtain that $v_i(1-\varphi_i(c_1))\ge 2$ because $v_i(1-\varphi_i(c_1))=v_i(\varphi_i(1-c_1))$, $1-c_1\in \tfrac{\omega}{a}$, and $1\not\in v_i(\tfrac{\omega}{a})$. Therefore, by noting that \eqref{eqaaa} can be rewritten as
\[
f_i^n+\varphi_i(k) f_i^{n+1}= (f_i^n+\varphi_i(k')f_i^{n+1}) - [1 - \varphi_i(c_1)][f_i^n+\varphi_i(k')f_i^{n+1}] + \varphi_i(g), 
\]
we can replace $g$ with $g-(1-c_1)(t_i^n+k't_i^{n+1})$. Hence, we can further assume that $\varphi_i(c_1)=1$ in \eqref{eqaaa}. 
Then, we have either $\varphi_i(k)=\varphi_i(k')$ or $v_i(\varphi_i(k)- \varphi_i(k'))>0$. Assume that $k\ne k'$.
By noting that the restriction $\varphi_i|_K: K \to R_i$ of a ring homomorphism $\varphi_i$ is injective, we obtain that $v_i(\varphi_i(k)- \varphi_i(k'))>0$. Since $k, k'\in K$, we have $k-k'\in K$. Hence, $k-k'\ne 0$ is a unit of $R$. It follows that $v_i(\varphi_i(k-k'))=0$. This is a contradiction. Therefore, $k=k'$ as desired.
\end{proof}

Now, we obtain a necessary condition for $R$ to have finite trace ideals. By noting that $R=\tfrac{\omega}{a}$ if and only if $R$ is Gorenstein, our result generalizes one direction of \cite[Corollary 2.16]{HM}.

\begin{thm}\label{mainthm2}
Let $(R, \fkm)$ be a one-dimensional Cohen-Macaulay complete local ring containing an infinite field $K$ as a subring. Suppose that $\regTr(R)$ is finite. Then, we have the following:
\begin{enumerate}[{\rm (a)}] 
\item $\ol{R}$ is finitely generated as an $R$-module.
\item There exist an $\fkm$-primary ideal $\omega$ and $a\in \omega$ such that $\omega\cong \omega_R$ and $\omega^{n+1}=a\omega^n$.
\item Let $\ol{R}\cong R_1\times \cdots \times R_\ell$, where $R_i$ are discrete valuation rings. Then, for each $1\le i \le \ell$, the value set of $\tfrac{\omega}{a}$ with respect to $R_i$ satisfies one of the following:
\begin{enumerate}[{\rm (i)}] 
\item $1\in v_i(\tfrac{\omega}{a})$.
\item $1\not\in v_i(\tfrac{\omega}{a})$ and for all $n\ge 2$, $n\in v_i(\tfrac{\omega}{a})$ or $n+1\in v_i(\tfrac{\omega}{a})$ holds.
\end{enumerate}
\end{enumerate}
\end{thm}

\begin{proof} 
(a) follows from Lemma \ref{lem2}. By noting that (a) is equivalent to saying that $R$ is reduced, there exists a canonical ideal (see \cite[4.6 Theorem]{LW} and \cite[Proposition 3.3.18]{BH}). This proves (b) because $R/\fkm$ is infinite.

(c): Note that for any birational extension $A$, $R:A$ is a trace ideal because $R:A=(R:~A)A=\tr_R(A)$. Hence, $\regTr(R)$ includes the set 
\[
\left\{ R:A\ \mid \ \text{$A$ is a birational extension} \right\}.
\]
Hence, for each $1\le i \le \ell$, $v_i(\tfrac{\omega}{a})$ must contain either $1, n$, or $n+1$ for all $n\ge 2$ by Proposition \ref{keysec4}. This completes the proof.
\end{proof}

Theorem \ref{mainthm2} provides the finiteness of indecomposable maximal Cohen-Macaulay $\varphi_i(R[\frac{\omega}{a}])$-modules.

\begin{cor}
With the same assumptions and notation of {\rm Theorem \ref{mainthm2}}, $\varphi_i(R[\frac{\omega}{a}])$ has finite indecomposable maximal Cohen-Macaulay modules for all $1\le i \le \ell$.
\end{cor}

\begin{proof} 
Since $R[\frac{\omega}{a}]$ is a birational extension, $\varphi_i(R[\frac{\omega}{a}])$ is a numerical semigroup (note that the ring $R[\frac{\omega}{a}]$ is independent of the choice of $\omega$ and $a$ by \cite[Theorem 2.5(3)]{CGKM} or \cite[Proposition 2.4(b)]{K}). Hence, the condition (ii) of Theorem \ref{mainthm2}(c) proves that for all $1\le i \le \ell$, $v_i(R[\frac{\omega}{a}])$ is either 
\begin{center}
$\langle 1\rangle$, $\langle 2, 2s-1\rangle$, $\langle 3, 4\rangle$, $\langle 3, 5\rangle$, $\langle 3, 5, 7\rangle$, or $\langle 3, 5, 7\rangle$, 
\end{center}
where $s$ is a positive integer. This proves that $R[\frac{\omega}{a}]$ has finite indecomposable maximal Cohen-Macaulay modules by \cite[4.10 Theorem]{LW}.
\end{proof}

Let us note an observation to give examples regarding Theorem \ref{mainthm2}. 

\begin{obs}
Let $R=K[[H]]$ be a numerical semigroup ring of $H$, and let $I$ be an ideal in $R$ such that $R:\ol{R}\subseteq I$. Then we may compute a system of generators of $R:I$ in the following way:

Set $I=R:\ol{R}+\langle f_1, f_2, \dots, f_\ell \rangle$ and write $f_i=\sum_{j\ge 0} a_{ij} t^j$ for $1\le i \le \ell$, where $a_{ij}\in K$. Then 
\[
R:I=R:(R:\ol{R})\cap \left( \bigcap_{i=1}^\ell R:f_i\right) = \ol{R}\cap \left( \bigcap_{i=1}^\ell R:f_i\right).
\]
Hence, for an element $g\in \rmQ(R)$, $g\in R:I$ if and only if 
\begin{center}
$\displaystyle g= \sum_{j\ge 0} b_j t^j \quad (b_j\in K)$ \quad such that \quad \\
  $\displaystyle gf_i=\sum_{j\ge 0} (a_{i0}b_j + a_{i1}b_{j-1}+\cdots +a_{ij}b_0) t^j\in R$ for all $1\le i \le \ell$. 
\end{center}
Since $R$ is a numerical semigroup ring, the condition $gf_i\in R$ can be checked by $a_{i0}b_j + a_{i1}b_{j-1}+\cdots +a_{ij}b_0=0$ for all $j\not\in H$. Because $\NN_0\setminus H$ is finite, we may compute a system of generators of $R:I$.
\end{obs}

\begin{ex}\label{ex4.6}
Let $R=K[[H]]$ be a numerical semigroup ring of $H$ over a field $K$. Let $\fkm$ be the maximal ideal of $R$. Set $\fkc=R:\ol{R}$.
\begin{enumerate}[{\rm (a)}] 
\item Let $H=\langle 4, 5, 11\rangle$. Then $\tfrac{\omega}{a}=\langle 1, t\rangle$ and 
\[
\Tr(R)=\{0, \fkc, \fkc+(t^5), \fkm=\fkc+(t^4, t^5), R\}.
\]
\item Let $H=\langle 4, 6, 9, 11\rangle$. Then $\tfrac{\omega}{a}=\langle 1, t^2, t^5\rangle$ and 
\[
\Tr(R)=\{0, \fkc, \fkc+(t^6), \fkm=\fkc+(t^4, t^6), R\}.
\]
\item Let $H=\langle 4, 5, 7\rangle$. Then $\tfrac{\omega}{a}=\langle 1, t^3\rangle$ and 
\[
\Tr(R)=\{0, \fkc, \fkc+(t^5), \fkm=\fkc+(t^4, t^5), R\}.
\]
\end{enumerate}
\end{ex}

\begin{proof}
(a): By \cite[Example (2.1.9)]{GW}, we have $\omega_R\cong \langle 1, t\rangle$. $\Tr(R)\supseteq \{0, \fkc, \fkm, R\}$ is clear. $\fkc+(t^5) \in \Tr(R)$ follows from the fact that $\fkc+(t^5)=t^5\ol{R} \cap R$ is an integrally closed ideal containing the conductor of $\ol{R}$ (see \cite[Proposition 2.6]{HM}).

Let $I\in \Tr(R)$ such that $I\not\in \{0, \fkc, \fkc+(t^5), \fkm, R\}$. Then, because $\fkc\subsetneq I\subsetneq \fkm$ and $\ell_R(\fkm/\fkc)=2$, $I=\fkc+(at^4+bt^5)$ for some $a, b\in K$. We may assume that $a=1$. Then, by Observation, we obtain that $g\in R:I$ if and only if $g=\sum_{j\ge 0} a_j t^j$ ($a_j\in K$) such that $ba_1+a_2=0$ and $ba_2+a_3=0$. Hence, 
\[
R:I=\langle 1, t-bt^2+b^2t^3, t^4, t^5, t^6, t^7\rangle.
\] 
It follows that $(t-bt^2+b^2t^3)I\subseteq (R:I)I=I$. In particular, $t^5+b^3t^8=(t-bt^2+b^2t^3)(t^4+bt^5)\in I$. This implies that $t^5\in I$ since $t^8\in \fkc\subseteq I$. Therefore, $t^4=(t^4+bt^5)-bt^5\in I$. It follows that $I=\fkm$, which is a contradiction. 

(b) and (c): They proceed in the same way as (a).
\end{proof}

\begin{rem}
We have no example such that the value set of $\tfrac{\omega}{a}$ satisfies one of Theorem \ref{mainthm2}(c)(i)-(ii) but $R$ has infinite trace ideals. 
\end{rem}

\section{The correspondence between $\Tr(R)\setminus \{R\}$ and $\Tr(\Hom_R(\fkm, \fkm))$}\label{section4.5}

In this section, we focus on the case that $(R, \fkm)$ is a one-dimensional Cohen-Macaulay local ring having minimal multiplicity. 
Then, we have the one-to-one correspondence between $\Tr(R)\setminus \{R\}$ and $\Tr(\Hom_R(\fkm, \fkm))$. As an application, we obtain a sufficient condition for the finiteness of trace ideals. 

\begin{thm} \label{main3}
Let $(R, \fkm)$ be a one-dimensional Cohen-Macaulay local ring having minimal multiplicity. Suppose that $R$ is not a discrete valuation ring. Choose $a\in \fkm$ such that $\fkm^2=a\fkm$ {\rm(}we need not assume that $R/\fkm$ is infinite, see \cite[Corollary 1.10]{Lip}{\rm )}. Set $B=\fkm:\fkm$. Then, the map
\[
\Tr(R)\setminus \{R\} \to \Tr(B);\quad I\mapsto \tfrac{I}{a}
\]
is bijective. 
\end{thm}

\begin{proof} 
(well-definedness): Let $I\in \Tr(R)\setminus \{R\}$. For all $\alpha\in B=\fkm:\fkm$, we have 
\[
\iota'\circ \hat{\alpha}\circ \iota: I\xrightarrow{\iota} \fkm \xrightarrow{\hat{\alpha}} \fkm \xrightarrow{\iota'} R \in \Hom_R(I, R),
\] 
where $\hat{\alpha}$ denotes the multiplication map by $\alpha$ and $\iota$, $\iota'$ are the inclusions. Hence, $\alpha I= \Im (\iota'\circ \hat{\alpha}\circ \iota)\subseteq I$ since $I$ is a trace ideal in $R$. It follows that $\tfrac{I}{a} B=\tfrac{I}{a}$. On the other hand, since $\fkm^2=a\fkm$, we have $B=\tfrac{\fkm}{a}$. It follows that $\tfrac{I}{a}\subseteq B$. Therefore, $\tfrac{I}{a}$ is an ideal in $B$.

Let $f\in \Hom_B(\tfrac{I}{a}, B)$. $f(I)=f(a\tfrac{I}{a})=af(\tfrac{I}{a})$. Because $f$ is an $R$-linear homomorphism and $I$ is a trace ideal in $R$, $f(I)\subseteq I$. Hence, we obtain that $f(\tfrac{I}{a})=\tfrac{f(I)}{a}\subseteq \tfrac{I}{a}$. It follows that $\tfrac{I}{a}$ is a trace ideal in $B$. 

(injectivity): This is clear.

(surjectivity): Let $J\in \Tr(B)$. We will prove that $aJ\in \Tr(R)\setminus \{R\}$. Let $f\in \Hom_R(aJ, R)$, and define the map 
\[
g: J\to B; \quad x\mapsto \tfrac{f(ax)}{a}
\]
for $x\in J$. 
\begin{claim} \label{claim3}
$g$ is a $B$-linear homomorphism.
\end{claim}
\begin{proof}[Proof of Claim \ref{claim3}] 
Note that $g$ is well-defined. Indeed, if $f(aJ)\not\subseteq \fkm$, then $f(aJ)=R$. This proves that $aJ$ has a free direct summand. In other words, $aJ=(z)+L$ and $(z)\cap L=0$ for some non-zerodivisor $z$ of $R$ and some ideal $L$ in $R$. $L=0$ since $zL\subseteq (z)\cap L=0$. Thus $aJ=zR$. It follows that $zB=aJB=aJ=zR$; hence, $R=B$. This is a contradiction for the assumption that $R$ is not a discrete valuation ring. Hence, $g(x)=\tfrac{f(ax)}{a}\in \tfrac{\fkm}{a}=B$ for all $x\in J$.

We prove that $g$ is a $B$-linear homomorphism. Let $\alpha\in B$ and $x\in J$, and write $\alpha = \tfrac{b}{a}$ for some $b\in \fkm$. Then, we obtain that 
\begin{align*} 
ag(\alpha x) = a\frac{f(a\alpha x)}{a} = \frac{f(a^2\alpha x)}{a} = \frac{f(ab x)}{a} =  b\frac{f(a x)}{a} = bg(x) = a \alpha  g(x).
\end{align*}
By noting that $a$ is a non-zerodivisor of $B$, it follows that $\alpha g(x)=g(\alpha x)$. 
\end{proof}

By Claim \ref{claim3}, $g(J)\subseteq J$ because $J$ is a trace ideal in $B$. It follows that $f(aJ)=ag(J)\subseteq aJ$. Hence, $aJ$ is a trace ideal in $R$.
\end{proof}

\begin{cor} \label{cor4.8}
Suppose that $(R, \fkm)$ is a one-dimensional Cohen-Macaulay local ring having minimal multiplicity. Then $\Tr(R)$ is finite if and only if $\Tr(\fkm:\fkm)$ is finite.
\end{cor}

\begin{rem}\label{remrem5.3}
Let us explain the meaning of $\fkm:\fkm$ in numerical semigroup rings. Let $e\ge 3$ be an integer and $R=K[[H]]$ a numerical semigroup ring containing $t^e$. Let $\fkm$ be the maximal ideal of $R$. Then, 
\[
\mathrm{Ap}_e(H)=\{h\in H \ \mid \ h-e\not\in H \}
\] 
is called the {\it Ap\'ery set} of $H$ with respect $e$ (\cite[before Lemma 2.4]{RS}). It is easy to obtain that $\mathrm{Ap}_e(H)=\{w_0(H)=0, w_1(H), \dots, w_{e-1}(H)\}$, where $w_i(H)$ is the least integer in $H$ congruent with $i$ modulo $e$. Write $w_i(H)=i+\mu_i(H)e$ for all $1\le i \le e-1$. 
We denote by $\calH_e$ the set of all numerical semigroups containing $e$, and denote by $\calP_e$ the cone
{\small
\begin{align*}
\left\{ (x_1, \dots, x_{e-1})\in \mathbb{R}^{e-1} \ \middle| \ \begin{matrix}
\text{for all $1 \le i \le j \le e-1$ with $i+j\ne e$, either}  \\
x_i + x_j \ge x_{i+j} \ \ (i+j<e) \text{ \ or \ } x_i + x_j \ge x_{i+j}-1 \ \  (i+j>e)
\end{matrix}
\right\}.
\end{align*}
}
With this notation, Kunz (\cite{Kun}) showed that there is a one-to-one correspondence between $\calH_e$ and $\calP_e\cap \mathbb{N}^{e-1}$ as follows:
\begin{align*}
\calH_e \to \calP_e\cap \mathbb{N}^{e-1}; \quad H\mapsto (\mu_1(H), \dots, \mu_{e-1}(H)).
\end{align*}
Furthermore, the interior of $\calP_e\cap \mathbb{N}^{e-1}$ corresponds to the numerical semigroups having minimal multiplicity $e$. With the above notation, if $R=K[[H]]$ corresponds to an interior point, then $\fkm:\fkm$ corresponds to $(\mu_1(H)-1, \dots, \mu_{e-1}(H)-1)$. 

In particular, Corollary \ref{cor4.8} claims that the finiteness of trace ideals in numerical semigroup rings with multiplicity $\le e$ can be reduced the boundary of $\calP_e\cap \mathbb{N}^{e-1}$.
\end{rem}

We further note an application of Theorem \ref{main3} to Arf rings. The notion of Arf rings originates from the classification of certain singular points of plane curves by Arf \cite{A}. Thereafter, Lipman \cite{Lip} generalized them by extracting the essence of the rings. The reader can consult \cite{IK, Lip} for several basic results on Arf rings. The examples arising from numerical semigroup rings are also found by using \cite[Corollary 3.19]{RS}. Here, we only note the characterization of Arf rings by Lipman. Let $R$ be a one-dimensional Cohen-Macaulay local ring. Set 
\[
\displaystyle R_1=\bigcup_{i\ge0}[\operatorname{J}(R)^i:\operatorname{J}(R)^i] \quad 
\text{and define recursively} \quad 
R_n=
\begin{cases}
\  R & \ \ (n=0)\\
\ [R_{n-1}]_1 & \ \ (n>0)\\
\end{cases}
\]
for each $n\ge 0$. Then, we obtain the tower
$$R=R_0\subseteq R_1\subseteq \cdots \subseteq R_n \subseteq \cdots$$
of rings in $\ol{R}$. Every $R_n$ is a Cohen-Macaulay semi-local ring such that all maximal ideals are of height one. 
By using this tower of rings, Lipman gave a characterization of Arf rings as follows.

\begin{fact}[{\cite[Theorem 2.2]{Lip}}]\label{4.9}
The following conditions are equivalent:
\begin{enumerate}[{\rm (a)}]
\item
$R$ is an Arf ring.
\item
$[R_n]_\fkM$ has minimal multiplicity for every $n\ge0$ and $\fkM\in\Max R_n$.
\end{enumerate}
When the equivalent conditions hold, $\overline{R}=\underset{n\ge0}{\bigcup}R_n$.
\end{fact} 

By using this characterization, we obtain an application to Arf rings.

\begin{cor} \label{cor4.11}
Let $(R, \fkm)$ be an Arf local ring. Suppose that $\ol{R}$ is finitely generated as an $R$-module and a local ring. Then, $\Tr(R)$ is finite.
\end{cor}

\begin{proof} 
By noting that all $R_n$ are local rings having minimal multiplicity and $\ol{R}/R$ is of finite length, the assertion follows from Corollary \ref{cor4.8} and Fact \ref{4.9}.
\end{proof}

\section{A question}\label{section5}

At the end of this paper, we note that the following question is left open.

\begin{quest}
Let $(R, \fkm)$ be an Noetherian local ring of dimension zero. Then, what is the property of rings having finite trace ideals?
\end{quest}

We only have the answer for Gorenstein rings (or rings having minimal multiplicity, see Theorem \ref{aa3.3}):

\begin{prop}
Let $K$ be an infinite field, and let $(R, \fkm)$ be a local $K$-algebra of dimension zero. Suppose that $R$ is Gorenstein, i.e., self-injective. Then $\Tr(R)$ is finite if and only if $R\cong S/\fkn^\ell$ for some discrete valuation ring $(S, \fkn)$ and some integer $\ell> 0$. 
\end{prop}

\begin{proof}
If $R\cong S/\fkn^\ell$, then there are only finite number of ideals in $R$. In particular, $\Tr(R)$ is finite. Suppose that $\Tr(R)$ is finite. Assume that $\fkm$ is not cyclic, and choose $x, y\in \fkm$ such that $x, y$ is a part of minimal generators of $\fkm$. Then $(x+ay)=(x+by)$ if and only if $a=b$ for $a, b\in K$. By recalling that all ideals are trace ideals in Artinian Gorenstein ring (\cite[Theorem 3.5]{LP}), this proves that $\Tr(R)$ is infinite. Hence, $\fkm$ is cyclic. Thus $R\cong S/\fkn^\ell$ for some discrete valuation ring $(S, \fkn)$ and some integer $\ell> 0$. 
\end{proof}




\end{document}